\documentclass{lms}
\usepackage{amssymb}
\usepackage{latexsym}
\usepackage{amsfonts}
\usepackage{amsmath}

\newtheorem{theorem}{Theorem}[section] 
\newtheorem{lemma}[theorem]{Lemma}     
\newtheorem{corollary}[theorem]{Corollary}
\newtheorem{proposition}[theorem]{Proposition}

\newnumbered{assertion}{Assertion}    
\newnumbered{conjecture}{Conjecture}  
\newnumbered{definition}{Definition}
\newnumbered{hypothesis}{Hypothesis}
\newnumbered{remark}{Remark}
\newnumbered{note}{Note}
\newnumbered{observation}{Observation}
\newnumbered{problem}{Problem}
\newnumbered{question}{Question}
\newnumbered{algorithm}{Algorithm}
\newnumbered{example}{Example}
\newunnumbered{notation}{Notation} 

\DeclareMathOperator{\base}{\mathfrak{B}}
\DeclareMathOperator{\topg}{\mathfrak{L}}

\DeclareMathOperator{\AGL}{AGL}

\DeclareMathOperator{\supp}{supp}

\DeclareMathOperator{\Perm}{Perm}

\DeclareMathOperator{\NR}{\mathcal{N}}
\DeclareMathOperator{\G}{\mathcal{G}}
\DeclareMathOperator{\reed}{\mathcal{R}}
\DeclareMathOperator{\cd}{\mathcal{C}}
\DeclareMathOperator{\PN}{\mathcal{PN}}
\DeclareMathOperator{\z}{{\bf{0}}}
\DeclareMathOperator{\1}{{\bf{1}}}
\DeclareMathOperator{\wt}{wt}
\DeclareMathOperator{\Aut}{Aut}
\DeclareMathOperator{\GL}{GL}

\def\cprime{$'$}

\newcommand{\bs}{\boldsymbol}
\newcommand{\Z}{\mathbb Z}
\newcommand{\F}{\mathbb F}

\title[Characterisations of the Nordstrom-Robinson codes]
      {New Characterisations of the\\ Nordstrom-Robinson codes}
\author{Neil I. Gillespie and Cheryl E. Praeger}

\classno{94B05, 05E18 (primary), 20B25, 05C25 (secondary)}

\extraline{This research was supported by the Australian Research
Council Federation Fellowship FF0776186 of the second author.} 

\begin{document}
\maketitle

\begin{abstract}
In his doctoral thesis, Snover proved that any binary $(m,256,\delta)$
code is equivalent to the Nordstrom-Robinson code or the punctured
Nordstrom-Robinson code for $(m,\delta)=(16,6)$ or $(15,5)$
respectively.  We prove that these codes are also characterised as
\emph{completely regular} binary codes with $(m,\delta)=(16,6)$ or $(15,5)$,
and moreover, that they are \emph{completely transitive}. 
Also, it is known that completely transitive codes are necessarily completely regular, but 
whether the converse holds has up to now been an open question.  We answer this by 
proving that certain completely regular codes are not completely transitive, namely, the (punctured) Preparata
codes other than the (punctured) Nordstrom-Robinson code.
\end{abstract}

\section{Introduction}

In \cite{hammons}, Hammons et al.~proved that certain interesting non-linear codes can be efficiently described
as the image under the Grey map of $\Z_4$-linear codes (see Section \ref{s:z4linear} for appropriate definitions).  
Their result has led to a significant research effort into 
$\Z_4$-linear codes; for various classifications and constructions, see, for example, \cite{borges1,calderbankkantor,Calderbank1,oneweight,fields1,fren1,krotov1,Pless1,solov1};
for interesting applications to steganography, see \cite{brier,jouhari}; for connections to unimodular lattices, respectively to semifield planes, see \cite{bonn1}, and \cite{semifields1,semifields2}; for a database of $\Z_4$-linear codes, see \cite{database}, and references within.

In their paper, Hammons et al. also gave an explanation to one of the outstanding problems in coding theory, that the
weight enumerators of the non-linear Kerdock codes and the Preparata codes satisfy the MacWilliams identities.
The first member of both of these families is the well known Nordstrom-Robinson code $\mathcal{N}$, which is a non-linear
$(16,256,6)$ binary code with several interesting properties. It is optimal, in the sense that it is the largest possible binary code of 
length $16$ with minimum distance $6$, and it is twice as large as any linear binary code with the same length and minimum distance.
Moreover, Snover \cite{snover} proved that any binary $(16,256,6)$ code is
equivalent to the Nordstrom-Robinson code. Analogous properties also hold for the punctured Nordstrom-Robinson code, a non-linear $(15,256,5)$ code.
In this paper, we prove that the Nordstrom-Robinson codes have other exceptional properties. First we prove 
that the codes are \emph{completely
  transitive}, and hence \emph{completely regular} (see 
Definition \ref{defs}). Then we show that binary completely regular codes with the same length and minimum distance
parameters are equivalent to the Nordstrom-Robinson codes.

\begin{theorem}\label{mainthm} Any binary completely regular code of length $m$ with minimum distance $\delta$
is equivalent to the Nordstrom-Robinson code, respectively the punctured Nordstrom-Robinson code, 
if $(m,\delta)=(16,6)$ or $(15,5)$.  Moreover, such a code is completely transitive.
\end{theorem}

It is known that completely transitive codes are
  necessarily completely regular \cite{giupra}.  A consequence
  of Theorem \ref{mainthm} is that the converse holds for binary codes with
  $(m,\delta)=(16,6)$ or $(15,5)$.  This is
  similar to a result in \cite{paphad} in which the authors proved
  that a binary completely regular code with $(m,\delta)=(12,6)$ or
  $(11,5)$ is unique up to equivalence, and that such codes are
  completely transitive. We demonstrate that the converse does not hold for any other code in an infinite family containing these two codes.

As mentioned above, the Nordstrom-Robinson code of length $16$ is the first member of a family of completely regular codes
called the Preparata codes (see \cite[Section 7.4.3]{vanlint} for a nice definition of the Preparata codes). 
It turns out that no other Preparata code is completely transitive, and similarly, no other punctured Preparata code apart from the punctured 
Nordstrom-Robinson code is completely transitive.
\begin{theorem}\label{main2}
The (punctured) Nordstrom-Robinson code is the only member of the (punctured) Preparata codes that is completely transitive.
In particular, other than the (punctured) Nordstrom-Robinson code, the (punctured) Preparata codes are completely regular but not completely transitive.
\end{theorem}
As far as the authors are aware, these are the first examples of completely regular codes shown not to be completely transitive.

In Section \ref{prelim}, we introduce the necessary definitions and
preliminary results.  Then in Section \ref{ctr} we prove that the
Nordstrom-Robinson code and the punctured Nordstrom-Robinson code are completely transitive, and we prove Theorem \ref{main2}.
We prove Theorem \ref{mainthm} in Section \ref{sec:proofmain1}.
In the final section we consider the natural question of 
whether the complete transitivity of the Nordstrom-Robinson code could be determined from the $\Z_4$-linear structure of 
its $\Z_4$-representation, the \emph{Octacode}. We give a discussion 
which suggests that the binary representation is the correct setting to prove that it is completely transitive.

\section{Definitions and Preliminaries}\label{prelim}

The \emph{binary Hamming graph $\Gamma=H(m,2)$} has vertex set 
$V(\Gamma)=\F_2^m$, the set of $m$-tuples with entries from the 
field $\F_2=\{0,1\}$, and an edge exists between two vertices if and
only if they differ in precisely one entry. The \emph{Hamming distance $d(\bs{\alpha},\bs{\beta})$} between $\bs{\alpha}, \bs{\beta}\in\F_2^m$ is the number
of entries in which the two vertices differ. 
Let $M=\{1,\ldots,m\}$, and view $M$ as the set of vertex entries of
$\Gamma$.  For $\bs{\alpha}\in\F_2^m$, the \emph{support of $\bs{\alpha}$} is
the set $\supp(\bs{\alpha})=\{i\in M\,:\,\alpha_i\neq 0\}$, and the
\emph{weight of $\bs{\alpha}$} is
$\wt(\bs{\alpha})=|\supp(\bs{\alpha})|$. 

A \emph{code} $C$ in $\Gamma$ is a non-empty subset of $V(\Gamma)$, and a \emph{codeword} is an element of $C$. 
The \emph{minimum distance, $\delta$, of C} is the smallest distance between distinct codewords 
of $C$. For any vertex $\bs{\gamma}\in\Gamma$, we define 
the \emph{distance of $\bs{\gamma}$ from $C$} to be
$$d(\bs{\gamma},C)=\min\{d(\bs{\gamma},\bs{\beta})\,|\,\bs{\beta}\in C\},$$
and the \emph{covering radius of $C$} to be
$$\rho=\max_{\gamma\in V(\Gamma)} d(\gamma,C).$$
We let $C_i$ denote the set of vertices that
are distance $i$ from $C$.  It follows that
$\{C=C_0,C_1,\ldots,C_\rho\}$ forms a partition of $V(\Gamma)$, called
the \emph{distance partition of $C$}. The \emph{distance distribution
  of $C$} is the $(m+1)$-tuple $a(C)=(a_0,\ldots,a_m)$ 
where $$a_i=\frac{|\{(\bs{\alpha},\bs{\beta})\in 
  C^2\,:\,d(\bs{\alpha},\bs{\beta})=i\}|}{|C|}.$$   
We observe that $a_i\geq 0$ for all $i$ and $a_0=1$.  Moreover,
$a_i=0$ for $1\leq i\leq \delta-1$ and $|C|=\sum_{i=0}^ma_i$.
In the Hamming graph, the \emph{MacWilliams transform} of the distance distribution of $C$, $a(C)$, is the $(m+1)$-tuple 
$a'(C)=(a_0',\ldots,a_m')$
where \begin{equation}\label{kracheqn}a'_k:=\sum_{i=0}^ma_iK_k(i)\end{equation}
with \begin{equation*}K_k(x):=\sum_{j=0}^k(-1)^j\binom{x}{j}\binom{m-x}{k-j}.\end{equation*}
It follows from \cite[Lemma 5.3.3]{vanlint}
that $a'_k\geq 0$ for $k\in\{0,1,\ldots,m\}$. 

The automorphism
group $\Aut(\Gamma)$ of the binary Hamming graph is semi-direct product $\base\rtimes\topg$ where $\base\cong S_2^m$ and
$\topg\cong S_m$, see \cite[Theorem 9.2.1]{distreg}.    Let
$g=(g_1,\ldots, g_m)\in \base$, $\sigma\in\topg$ and
$\bs{\alpha}=(\alpha_1,\ldots,\alpha_m)\in V(\Gamma)$.  Then $g\sigma$ acts
on $\bs{\alpha}$ in the following
way: 
\begin{equation}\label{eq:hamact}\bs{\alpha}^{g\sigma}=(\alpha_{1{\sigma^{-1}}}^{g_{1{\sigma^{-1}}}},\ldots,\alpha_{m{\sigma^{-1}}}^{g_{m{\sigma^{-1}}}}). 
\end{equation} Since the
base group $\base\cong S_2^m$ of $\Aut(\Gamma)$ acts regularly on
$V(\Gamma)$, we may identify $\base$ with the group of
translations of $\F^m_{2}$, and $\Aut(\Gamma)$ with a subgroup of the
affine group $\AGL(m,2)$.  More precisely $\base$ consists of the
translations $g_{\bs{\beta}}$, where
$\bs{\alpha}^{g_{\bs{\beta}}}=\bs{\alpha}+\bs{\beta}$ for 
$\bs{\alpha},\bs{\beta}\in \F_2^m$, and if $\z$ is the zero vector, then
$\Aut(\Gamma)=\base\rtimes \Aut(\Gamma)_{\mathbf{0}}$ where
$\Aut(\Gamma)_{\mathbf{0}}$ (the stabiliser of $\z$ in
$\Aut(\Gamma)$) is the group of permutation matrices in $\GL(m,2)$. 
The \emph{automorphism group of a code $C$}, $\Aut(C)$, is
the setwise stabiliser in $\Aut(\Gamma)$ of $C$. We let $\Perm(C)$
denote the group of permutation matrices that fix $C$ setwise.
We say two codes $C$ and $C'$ in $\Gamma$ are \emph{equivalent} if
there exists $x\in\Aut(\Gamma)$ such that $C^x=C'$.  

\begin{remark}  In traditional coding theory, only weight preserving
  automorphisms of a code are considered, and so in the binary case,
  $\Perm(C)$ is defined as the automorphism group of a code.
  Consequently, established results about automorphism groups of
  certain codes refer to $\Perm(C)$, not $\Aut(C)$.  However, if
  $\z\in C$ we note that $\Aut(C)_{\z}$ is equal to $\Perm(C)$.       
\end{remark}

\begin{definition}\label{defs}  Let $C$ be code in $\Gamma$ with
  distance partition $\{C,C_1,\ldots,C_\rho\}$, and $\bs{\gamma}\in C_i$.
  We say $C$ is \emph{completely regular} if $|\Gamma_{k}(\bs{\gamma})\cap
  C|$ depends only on $i$ and $k$, and not on the choice of $\bs{\gamma}\in
  C_i$.  If there exists $X\leq \Aut(\Gamma)$ such that each $C_i$
  is an $X$-orbit, then we say $C$ is \emph{$X$-completely
    transitive}, or simply \emph{completely transitive}.  
\end{definition}

\begin{lemma}\cite{neum}\label{lem:rcr} Let $C$ be a completely regular code in $\Gamma$ with distance
partition $\{C,C_1,\ldots,C_\rho\}$.  Then $C_\rho$ is a completely regular code with distance partition $\{C_\rho,C_{\rho-1},\ldots,C\}$.
\end{lemma}

If a code $C$ is a subspace of $\F_2^m$ with dimension $k$, we say
$C$ is a \emph{linear} $[m,k,\delta]$ code.  If $C$ is not a linear
code we say $C$ is a $(m,|C|,\delta)$ code, where $|C|$ denotes the
cardinality of $C$. A code is \emph{antipodal} if $\bs{\alpha}+\bs{1}\in C$ for all $\bs{\alpha}\in C$, where $\bs{1}=(1,\ldots,1)$, otherwise we say $C$ is
\emph{non-antipodal}.  

Let $\bs{\alpha}$, $\bs{\beta}$ be two 
vertices in $\F_2^m$.  Then we say $\bs{\alpha}$ is \emph{covered} by
$\bs{\beta}$ if for each non-zero component $\alpha_i$ of
$\bs{\alpha}$ it holds that $\alpha_i=\beta_i$. Let $\mathcal{D}$ be a set of vertices of weight $k$ in 
$\Gamma$. Then we say $\mathcal{D}$ is a \emph{$t$-$(m,k,\lambda)$ design} if for every vertex $\bs{\nu}$ of weight 
$t$, there exist exactly $\lambda$ vertices of $\mathcal{D}$ that cover $\bs{\nu}$. This definition coincides with the usual definition of a
$t$-$(m,k,\lambda)$ design (see \cite{camvan}, for example), in the sense that
the rows of the incidence matrix of a $t$-design are the elements of $\mathcal{D}$. 

We let $b$ denote the size of $\mathcal{D}$. If
$\mathcal{D}$ is a $t$-design, then it is also an $j-(m,k,\lambda_j)$ 
design for $0\leq j\leq t-1$ 
\cite[Corollary 1.6]{camvan}
where \begin{equation}\label{arith1}\lambda_j\binom{k-j}{t-j}=\lambda\binom{m-j}{t-j}.\end{equation}
Using this fact we can deduce
that \begin{equation}\label{arith2}\binom{m}{j}\lambda_j=b\binom{k}{j}.\end{equation}
For further concepts and definitions about $t$-designs see
\cite{camvan}.

Let $p\in M=\{1,\ldots m\}$, and $C$ be a code in $\F_2^m$.  By
deleting the same coordinate $p$ from each codeword of $C$, we
obtain a code in $\F_2^{m-1}$, which we call the \emph{punctured
code of $C$ with respect to $p$}. We can also think of this as the projection of $C$
onto $J=M\backslash\{p\}$. Indeed, for a general $J=\{i_1,\ldots,i_k\}\subseteq M$,
let $\pi_J:\F_2^m\longrightarrow\F_2^{|J|}$ denote the projection onto the entries in $J$, and 
define $\pi_J(C)=\{\pi_J(\bs{\alpha})\,:\,\bs{\alpha}\in C\}$. When we project we would like to have some
group information available to us.  We have an induced action of
$\Aut(\Gamma)_J=\{g\sigma\in \Aut(\Gamma)\,:\,J^\sigma=J\}$ as
follows: for $x\in \Aut(\Gamma)_J$, we define
\begin{equation}\label{chi}\begin{array}{c c c c}
\chi(x):&\F_2^{|J|}&\longrightarrow&\F_2^{|J|}\\   
&\pi_J(\bs{\alpha})&\longmapsto&\pi_J(\bs{\alpha}^x),\\
\end{array}\end{equation} and observe that
$\ker\chi=\{(g_1,\ldots,g_m)\sigma\in\Aut(\Gamma)_J\,:\,j^\sigma=j\textnormal{ and }g_j=1\textnormal{ for }j\in J\}.$    

The following is a consequence of a result proved by Van Tilborg
\cite[Thm. 2.4.7]{vantil}. (Earlier this was proved for uniformly packed codes in the narrow sense \cite{semzin71}.)  For a code $C$ and a positive integer
$k$ we denote by $C(k)$ the set of weight $k$ codewords of $C$.  

\begin{theorem}\label{des1} Let $C$ be a completely regular code in
  $\Gamma$ that contains the zero vertex.  Then for each $k$ with
  $\delta\leq k\leq m$ and $C(k)\neq\emptyset$, it holds that $C(k)$
  forms a $t$-design with $t=\lfloor\frac{\delta}{2}\rfloor$. 
\end{theorem}

\subsection{The Nordstrom-Robinson code $\NR$}  
The Nordstrom-Robinson code was discovered by Nordstrom and Robinson in \cite{nord}, and independently by Semakov and Zinoviev  in \cite{semzin69}. 
It is a binary, non-linear, $(16,256,6)$ code, and Snover proved that all binary $(16,256,6)$ codes are equivalent \cite{snover}. 
So if one desired, one can 
take the definition of the code to be any $(16,256,6)$ code. However, 
in order for us to prove the complete transitivity of the Nordstrom-Robinson code,
we require the following description due to Goethals \cite{goethals}.

Let $\G$ be the $[24,12,8]$ extended binary Golay code (defined, for example, in \cite[p.131]{camvan}), chosen so  
that $\bs{\bar{\gamma}}=(1^8,0^{16})\in\G$.  Let $J^*=\{1,\ldots,8\}$ and $J=M\backslash J^*$.
We define the following subcode of  
$\G$: $$\cd=\{\bs{\bar{\alpha}}\in\G\,:\,\supp(\bs{\bar{\alpha}})\cap 
J^*=\emptyset \}.$$ For $1\leq i\leq 7$, let
$\bs{\bar{\alpha}}_i$ be a codeword in $\G$ with
$\supp({\bs{\bar{\alpha}_i}})\cap J^*=\{i,8\}$ (such codewords exist
in $\G$, see \cite[p.73]{macwill}), and let $\cd^i$   
be the coset $\bs{\bar{\alpha}}_i+\cd$. It follows that $\cd^i$
consists of all the codewords $\bs{\bar{\alpha}}\in\G$ such that
$\supp(\bs{\bar{\alpha}})\cap J^*=\{i,8\}$.

\begin{definition}  Let $\mathcal{A}=\cup_{i=0}^7 \cd^i$, where
  $\cd^0=\cd$.  The \emph{Nordstrom-Robinson code} $\NR$ is defined to
  be $\mathcal{A}$ with the first $8$ coordinates deleted, that is, the projection code of $\mathcal{A}$ onto $J$.
\end{definition}

Berlekamp proved that $\NR$ is a binary $(16,256,6)$ code, and that
$\Aut(\NR)_{\z}=\Perm(\NR)=2^4:A_7$ acting $3$-transitively on $16$ points
\cite{berl}, where $\z$ is the zero codeword in $\NR$. We also require the following. It is also known that $\Perm(\G)\cong M_{24}$ \cite[Ch. 20]{macwill}, and hence
$\Aut(\G)=T_{\G}\rtimes\Perm(\G)$ where $T_{\G}$ is the group of
translations generated by $\G$ \cite{giupra}.  Furthermore, by
\cite[p.96]{atlas} \begin{equation*}H:=\Perm(\G)_{\bs{\bar{\gamma}}}\cong\AGL(4,2)\cong  
  2^4:A_8.\end{equation*} It follows that $H$ has an induced action on $J^*$ that is permutationally
isomorphic to $A_8$, and also a faithful action on $J$. Moreover, $H\leq\Perm(\cd)$.

Semakov and Zinoviev \cite{semzin69} showed that the Nordstrom-Robinson code can partitioned into the union of $8$ cosets
of the Reed-Muller code $R(1,4)$. Indeed, this can be seen in the above description. 
Let $\mathcal{R}$ be the subcode of $\NR$ equal to the
projection code of $\cd$ onto $J$, and for $i=1,\ldots,7$ let $\reed^i$ be
the projection code of $\cd^i$ onto $J$, so
$\NR=\bigcup_{i=0}^7\reed^i$ where $\reed=\reed^0$.  The code
$\mathcal{R}$ is the linear $[16,5,8]$ Reed Muller code $R(1,4)$ \cite[p.74]{macwill}, and it follows, for each
$i=1,\ldots,7$, that $\reed^i$ is a coset of $\reed$.

\subsection{On the Complete regularity of the (Punctured) Preparata codes}\label{sec:comreg}

The Nordstrom-Robinson code is the first member the \emph{Preparata codes} \cite{preparata}, an infinite family of non-linear binary codes. For each odd $k\geq 3$, the Preparata code $\mathcal{P}(k)$ has length $2^{k+1}$, contains $2^{k+1}-2(k+1)$ codewords and has minimum distance $6$ (see for example, \cite[Section 7.4.3]{vanlint}). The code $\mathcal{P}(3)$ is equivalent to the Nordstrom-Robinson code $\mathcal{N}$ of length $16$. 
It is well known that $\mathcal{P}(k)$ and the punctured Preparata code $\mathcal{P}(k)^*$, 
are both completely regular for all odd $k\geq 3$ (see, for example, \cite[Ex. 6.3]{sole}). 

\begin{remark}\label{rem:comreg} The complete regularity of the (punctured) Preparata codes can be deduced from earlier work of Semakov et al.~\cite{semzin71}.
They proved that the punctured Preparata codes are \emph{uniformly packed (in the narrow sense)} with covering radius $3$, which also implies 
that the Preparata codes have covering radius $4$.
They then showed that a uniformly packed code (in the narrow sense) $C$ with covering radius $\rho$ has
exactly $\rho+1$ different weight distributions amongst all translates of $C$  \cite[Thm. 4]{semzin71}, which is an equivalent definition of a completely regular code. 
A similar result for the Preparata codes can also be deduced from \cite[Thm. 5]{semzin71}. Alternatively, Bassalygo and Zinoviev proved that the Preparata codes are \emph{uniformly packed (in the wide sense)} \cite{baszin77}, and from this one can easily deduce that they are completely regular (see, for example, \cite[Lemma 2.3]{paphad}).
\end{remark}

\section{Complete transitivity of the Nordstrom-Robinson codes}\label{ctr}

Let $\Gamma=H(16,2)$, and recall that $\Aut(\NR)$ is the stabiliser of $\NR$ in $\Aut(\Gamma)$.
The following homomorphism defines an action of $\Aut(\NR)$ on
$M=\{1,\ldots,16\}$.  \begin{equation}\label{mu}\begin{array}{c c c c}
\mu:&\Aut(\NR)&\longrightarrow&S_{16}\\   
&g\sigma&\longmapsto&\sigma\\
\end{array}\end{equation} 
We let $K=\Aut(\NR)\cap\base$ denote the kernel of the map $\mu$. We note that since $\NR$ is the union of
cosets of $\reed$, the group $T_{\reed}$ of translations generated by $\reed$ is a subgroup of $K$.  

\begin{theorem}\label{nrctr} $\NR$ is completely transitive.   
\end{theorem}

\begin{proof}  We first prove that for each $\bs{\beta}\in\NR$, there
  exists an $x\in\Aut(\NR)$ such that $\bs{\beta}^x=\z$, and hence $\Aut(\NR)$
  acts transitively on $\NR$.  Let $\bs{\beta}\in\NR$.  If
  $\bs{\beta}\in\reed$ then as $T_{\reed}\leq K$, $g_{\bs{\beta}}\in\Aut(\NR)$, and it follows that
  $\bs{\beta}^{g_{\bs{\beta}}}=\bs{\beta}+\bs{\beta}={\bf{0}}$.  Now
  suppose that $\bs{\beta}\in\NR\backslash\reed$, and let $\bs{\bar{\beta}}$ be the codeword in $\mathcal{A}$
  that projects onto $\bs{\beta}$. Then there
  exists a unique $i\in\{1,\ldots,7\}$ such that
  $\bs{\bar{\beta}}\in\cd^i=\bs{\bar{\alpha}}_i+\cd$.  
  Let $g$ be the translation of $\F_2^{24}$ generated by
  $\bs{\bar{\alpha}}_i$, let $\sigma\in H$ such that $i^\sigma=8$,
  and let $x=g\sigma\in\Aut(\G)$. We claim that $\chi(x)\in\Aut(\NR)$, where $\chi$ is as in (\ref{chi}).  

  Since $\sigma\in\Perm(\cd)$ it follows that
  $(\mathcal{C}^i)^x=(\bs{\bar{\alpha}}_i+\bs{\bar{\alpha}}_i+\cd)^\sigma=\cd^\sigma=\cd$.
  In particular, $\bs{\bar{\beta}}^x\in\cd$.  Furthermore,
  $\cd^x=(\bs{\bar{\alpha}}_i+\cd)^\sigma=\bar{\bs{\alpha}}_i^\sigma+\cd$. 
  Now, because $\supp(\bs{\bar{\alpha}}_i^\sigma)\cap
  J^*=\{i^\sigma,8^\sigma\}=\{8,8^\sigma\}$ and $\sigma$ stabilises
  $J^*$, it follows that
  $\bs{\bar{\alpha}}_i^\sigma\in\bs{\bar{\alpha}}_{8^\sigma}+\cd$ and so
  $\cd^x=\bs{\bar{\alpha}}_{8^\sigma}+\cd=\mathcal{C}^{k}$, where $k=8^\sigma$.  Now, for $j\neq i$ or
  $0$, consider $\cd^j=\bs{\bar{\alpha}}_j+\cd$.  Then
  $(\bs{\bar{\alpha}}_j+C)^x=(\bs{\bar{\alpha}}_j+\bs{\bar{\alpha}}_i+\cd)^\sigma=(\bs{\bar{\alpha}}_j+\bs{\bar{\alpha}}_i)^\sigma+\cd$,   
  because $\sigma\in\Perm(\cd)$.  It follows that
  $\supp(\bs{\bar{\alpha}}_j+\bs{\bar{\alpha}}_i)\cap J^*=\{j,i\}$,
  and so $\supp((\bs{\bar{\alpha}}_j+\bs{\bar{\alpha}}_i)^\sigma)\cap
  J^*=\{j^\sigma,i^\sigma\}=\{j^\sigma,8\}$.  Consequently
  $(\bs{\bar{\alpha}}_j+\bs{\bar{\alpha}}_i)^\sigma\in
  \bs{\bar{\alpha}}_{j^\sigma}+\cd$, and so 
  $(\mathcal{C}^j)^x=\bs{\bar{\alpha}}_{j^\sigma}+\cd=\cd^{\ell}$, where $\ell=j^\sigma$.
  Hence $x$ fixes setwise $\mathcal{A}$.  Because $\NR=\pi_J(\mathcal{A})$ we deduce
  that $\chi(x)\in\Aut(\NR)$.   

  Since $\bs{\bar{\beta}}^x\in\cd$, there exists
  $\bs{\eta}\in\reed$ such that
  $\pi_J(\bs{\bar{\beta}}^x)=\bs{\eta}$.  As $T_{\reed}\leq K$, $g_{\bs{\eta}}\in K$, and so $y=\chi(x)g_{\bs{\eta}}\in\Aut(\NR)$, and
  we have, by
  (\ref{chi}), \begin{equation*}\bs{\beta}^y=\pi_J(\bs{\bar{\beta}})^{\chi(x)g_{\bs{\eta}}}=\pi_J(\bs{\bar{\beta}}^x)^{g_{\bs{\eta}}}=\bs{\eta}^{g_{\bs{\eta}}}  
  =\bs{\eta}+\bs{\eta}={\bf{0}}.\end{equation*}  Consequently,
  $\Aut(\NR)$ acts transitively on $\NR$.    
  
  Recall that $\NR$ has covering radius $\rho=4$ (Remark \ref{rem:comreg}). Let $\NR_i$ denote the set of vertices at distance $i$ from $\NR$ for $i=1,\ldots,4$. 
  Since $\Aut(\NR)_{{\bf{0}}}\cong 2^4:A_7$ is acting 
  $3$-transitively on entries and $\delta=6$, one deduces that $\Aut(\NR)_{\z}$ acts transitively on
  $\Gamma_{i}(\z)=\Gamma_{i}(\z)\cap \NR_i$ for $i=1,2,3$.  Hence, by \cite[Lemma
    2.2]{paphad}, $\Aut(\NR)$ acts transitively on $\NR_i$ for
  $i=1,2,3$.       

  Let $\bs{\nu}$ be an element of $\NR_4$.
  Then there exists $\bs{\alpha}\in\NR$ such that $d(\bs{\nu},\bs{\alpha})=4$. As $\Aut(\NR)$
  acts transitively on $\NR$, there exists $x\in\Aut(\NR)$ such that $\bs{\alpha}^x=\z$.
  In particular $d(\z,\bs{\nu}^x)=4$, and because $\Aut(\NR)$ preserves the distance partition of $\NR$,
  it follows that $\Gamma_{4}(\z)\cap\NR_4\neq\emptyset$.  Also, let
  $\bs{\beta}$ be any codeword of weight $6$, and let
  $\bs{\nu}^*\in\Gamma_{4}(\z)$ be such that
  $\supp(\bs{\nu}^*)\subseteq\supp(\bs{\beta})$.  Then 
  $d(\bs{\nu}^*,\bs{\beta})=2$, and so
  $\Gamma_{4}(\z)\cap\NR_2\neq\emptyset$.  Consequently, because
  $\Aut(\NR)_{\z}$ fixes setwise $\Gamma_{4}(\z)$ and preserves
  the distance partition of $\NR$, $\Aut(\NR)_{\z}$ has at least $2$
  orbits on $\Gamma_{4}(\z)$.  Moreover, we see in \cite[Table XI]{orbit}
  that $\Aut(\NR)_{\z}$ has exactly two orbits on
  $\Gamma_{4}(\z)$.  Thus $\Aut(\NR)_{\z}$ acts transitively on
  $\Gamma_{4}(\z)\cap\NR_4$, and so, by \cite[Lemma 2.2]{paphad}, 
  $\Aut(\NR)$ acts transitively on $\NR_4$.      
\end{proof}

\begin{corollary}\label{cor1} $K=T_\mathcal{R}$ and $\Aut(\NR)/K\cong 2^4:A_8$.
\end{corollary}

\begin{proof} Since $\Aut(\NR)$ acts transitively on $\NR$, and $\Aut(\NR)_{\z}\cong 2^4:A_7$, we have that $|\Aut(\NR)|=2^{12}|A_7|$. 
Moroever, $\Aut(\NR)/K$ is a $3$-transitive subgroup of $S_{16}$ containing $2^4:A_7$, and so, by the
classification of finite $2$-transitive groups of degree $16$ (see \cite{sim}, for example), $\Aut(\NR)/K\cong 2^4:A_7$, $2^4:A_8$, $A_{16}$ or
$S_{16}$. As $T_{\reed}\leq K$, the only possibility is that $K=T_{\reed}$ and $\Aut(\NR)/K\cong 2^4:A_8$.
\end{proof}

\subsection{The Punctured Nordstrom-Robinson Code $\PN$}

The punctured Nordstrom-Robinson code $\PN$ is a $(15,256,5)$ code (see, for example, \cite{preparata}). 
Moreover, since all $(15,256,5)$ codes are equivalent, we can assume without loss of generality
that $\PN$ is obtained from $\mathcal{N}$ by puncturing the first entry, as in \cite{berl}. 
Recall also (Remark \ref{rem:comreg}) that $\PN$ has covering radius $3$.
By \cite[Lemma 6.5]{berl}, $\Aut(\PN)_{\z}\cong A_7$ acting $2$-transitively on
$15$ points. The action of $\Aut(\PN)_{\z}\cong A_7$ on
$\Gamma_{3}(\z)$ is equivalent to its action on the $3$-element subsets
of $M=\{1,\ldots,15\}$. The permutation characters for actions of
$A_7$ on $M$, and
on the $3$-element subsets of $M$, have inner product equal to $2$, see 
\cite[p.10]{atlas}.  Hence $A_7$ has exactly two orbits on $3$-element
subsets of $M$, so $\Aut(\PN)_{\z}$ has exactly two orbits on $\Gamma_{3}(\z)$.   

\begin{theorem}\label{pnrctr} $\PN$ is completely transitive.  
\end{theorem}

\begin{proof} Let $M$ denote the set of entries of $\NR$ and $J=M\backslash\{1\}$. Recall the homomorphism $\chi$
  from (\ref{chi}) with kernel equal to $\ker\chi=\langle
  (g_1,\ldots,g_{16})\rangle$, where $g_1=(0\,1)$ and $g_i=1$
  for $i\neq 1$. Also note that $\Aut(\NR)_J$ is equal to $\Aut(\NR)_{\{1\}}$ because $J$ and
 $\{1\}$ are disjoint sets. Now, since $\Aut(\NR)\cap\base=K=T_{\reed}$ (Corollary
  \ref{cor1}), it follows that $\Aut(\NR)\cap\ker\chi=1$.  
  Hence $\chi(\Aut(\NR)_{\{1\}})\cong\Aut(\NR)_{\{1\}}$, and it is straightforward to
  show that $\chi(\Aut(\NR)_{\{1\}})\leq\Aut(\PN)$.  Also
  $K\leq\Aut(\NR)_{\{1\}}$, and $\Aut(\NR)/K\cong 2^4:A_8$, by Corollary
  \ref{cor1}.  Thus $\Aut(\NR)_{\{1\}}/K\cong A_8$.  As $\Aut(\PN)_{\z}\cong A_7$, it follows
  from the orbit stabiliser theorem that $$|\Aut(\PN)|\leq  
  |\PN||\Aut(\PN)_{\z}|=|K||A_8|=|\Aut(\NR)_{\{1\}}|.$$ Hence, we deduce
  that $\Aut(\NR)_{\{1\}}\cong\Aut(\PN)$ and $\Aut(\PN)$ acts transitively
  on $\PN$.            

  As stated above, $\Aut(\PN)_{\z}$ has exactly two orbits on $\Gamma_3(\z)$. Thus,  by following a similar argument to the one used the proof of Theorem \ref{nrctr}, 
  recalling that $\PN$ has coving radius $3$ and minimum distance $5$, one deduces that $\Aut(\PN)$ acts transitively on $\PN_i$ for $i=1,2,3$.  
\end{proof}

Recall from Section \ref{sec:comreg} that $\mathcal{N}$ is the first member the Preparata codes \cite{preparata}. 
Theorem \ref{main2} follows from the next result.

\begin{proposition}\label{p:thmproof} The Preparata code $\mathcal{P}(k)$ and the punctured Preparata Code $\mathcal{P}(k)^*$ are completely transitive if and only if $k=3$. In particular, for $k>3$ odd, $\mathcal{P}(k)$ and $\mathcal{P}(k)^*$ are completely regular but not completely transitive.
\end{proposition}

\begin{proof} First, let us consider $\Perm(\mathcal{P}(k))$ and 
$\Perm(\mathcal{P}(k)^*)$,  the group of permutation matrices that fix the respective code setwise. In \cite{kantor}, Kantor showed that, for odd $k>3$, 
$\Perm(\mathcal{P}(k))$ acts imprimitively on entries and $\Perm(\mathcal{P}(k)^*)$ has order $(2^k-1)k$.  However, it is known that for any binary completely transitive 
code $C$ of length $m$ with minimum distance at least $5$, the group $\Perm(C)$ acts $2$-homogeneously on entries \cite[Prop. 2.5]{entryfaith}. Therefore $\Perm(C)$ 
acts primitively on entries and $\binom{m}{2}$ divides $|\Perm(C)|$. By combining this result with Kantor's results, and recalling that $m=2^{k+1}$ or $2^{k+1}-1$, 
we deduce that $\mathcal{P}(k)$ and $\mathcal{P}(k)^*$ are not completely 
transitive for $k>3$. The backwards implication of the statement is a consequence of Theorem \ref{mainthm}. The complete regularity of $\mathcal{P}(k)$ and $\mathcal{P}(k)^*$ for all odd $k\geq 3$ is well known (see Section \ref{sec:comreg}), which proves the final statement.
\end{proof}

\section{Proof of Theorem \ref{mainthm}}\label{sec:proofmain1}

  Let $\Gamma=H(m,2)$ and $C$ be a completely regular code in $\Gamma$ with minimum
  distance $\delta$ for $(m,\delta)=(16,6)$ or $(15,5)$.  Complete
  regularity and minimum distance are preserved by equivalence,
  therefore, by replacing $C$ with an equivalent code if necessary, 
  we can assume that $\z\in C$.   Since $C$ contains $\z$ and is
  completely regular, it follows that $C(\delta)\neq\emptyset$, where
  $C(\delta)$ is the set of codewords of weight $\delta$.  Hence, by
  Theorem \ref{des1}, $C(\delta)$ forms a $t$-$(m,\delta,\lambda)$
  design for $t=\lfloor\frac{\delta}{2}\rfloor$ 
  and some positive integer $\lambda$.  Using \eqref{arith1} with $j=1$
  in the case $(16,6)$ and \eqref{arith2} with $j=2$ in the case $(15,5)$, we deduce that
  $2$ divides $\lambda$. Let $S$ be the set 
  of $\bs{\alpha}\in C(\delta)$ such that
  $\{1,\ldots,t\}\subset\supp(\bs{\alpha})$.  It follows that
  $|S|=\lambda$, and as $C$ has minimum distance $\delta$, we deduce
  that $\supp(\bs{\alpha})\cap\supp(\bs{\beta})=\{1,\ldots,t\}$ for all distinct
  pairs of codewords $\bs{\alpha},\bs{\beta}\in S$.  Consequently, a simple
  counting argument gives that $$\lambda\leq\frac{m-t}{\delta-t}.$$
  In both cases we deduce that $\lambda<5$, so $\lambda=2$ or $4$.
  However, by Line 21 of \cite[Table 3.37]{handbk} and Line 16 of
  \cite[Table 1.28]{handbk}, it follows that a $t$-$(m,\delta,\lambda)$ 
  design does not exist in both cases for $\lambda=2$.  Thus $\lambda=4$.  

  {\underline{\bf{Case $(m,\delta)=(16,6)$:}}}  In this case $C(6)$ forms a $3$-$(16,6,4)$ design, which is therefore also a $j-(16,6,\lambda_j)$ 
  design for $j\leq 3$, and in particular, $\lambda=\lambda_3=4$, $\lambda_2=14$,
  $\lambda_1=42$ and $\lambda_0=112$.  Let $\bs{\beta}\in C(6)$ and define $n_i=|\{\bs{\gamma}\in C(6)\,:\,|\supp(\bs{\gamma})\cap\supp(\bs{\beta})|=i\}|$.    
  Because $C(6)$ is necessarily a simple design, it follows that $n_6=1$, 
  and since $\delta=6$ we deduce that $n_4=n_5=0$. Then, by applying \cite[Thm. 5]{semzin69balcodes}, we deduce
  that $n_3=60$, $n_2=15$, $n_1=36$ and $n_0=0$.   
  Because $n_1\neq 0$, it follows that $\Gamma_{10}(\bs{\beta})\cap C\neq\emptyset$, 
  and therefore, because $C$ is completely regular, $C(10)\neq\emptyset$.  \emph{We now
  claim that $\1\in C$, and consequently, that $C$ is antipodal.}  Suppose to the contrary.  
  Then $\1+C_{\rho}=C$ and $\rho\geq\delta-1=5$, where $\rho$ is the covering 
  radius of $C$ \cite[Thm. 11]{nonantipodal}.  Moreover, because $C(10)\neq\emptyset$, it holds that $\rho\leq 6$, and 
  because $\delta=6$, it follows from Lemma \ref{lem:rcr} that $\1+C_{\rho-i}=C_i$ for $i=1,2$.  
  We now calculate the size of the set $C_3$ in the distance partition of $C$.  To do this, 
  we count the pairs $\{(\bs{\nu},\bs{\gamma})\in C_3\times C\,:\,d(\bs{\nu},\bs{\gamma})=3\}$.  Counting this set in two 
  ways gives $$|C_3||\Gamma_3(\bs{\nu})\cap C|=|C|\binom{16}{3},$$ where $\bs{\nu}$ is any vertex in $C_3$.
  Fix $\bs{\nu}\in\Gamma_3(\z)$.  It follows that that if $\bs{\gamma}\in\Gamma_3(\bs{\nu})\cap C$ then either $\bs{\gamma}=\z$ or
  $\bs{\gamma}\in C(6)$, and if the later holds then $\bs{\gamma}$ covers $\bs{\nu}$.  Therefore, because $C(6)$ forms a $3$-$(16,6,4)$
  design, we deduce that $|\Gamma_3(\bs{\nu})\cap C|=5$, and so $|C_3|=|C|\times 112$.  Now suppose that $\rho=5$.  Then, by Lemma \ref{lem:rcr}, 
  $|C_2|=|C_3|$.  However, $|C_2|=|C|\binom{16}{2}$ which is a contradiction.  Thus $\rho=6$.  However, then Lemma \ref{lem:rcr} implies that
  $|C|(2+2\times 16+2\times\binom{16}{2}+112)=2^{16}$, which is a contradiction.  Hence $\1\in C$ and $C$ is antipodal.    
  Therefore, if $a(C)=(a_0,\ldots,a_m)$, we deduce that $a_i=a_{m-i}$ for all $i$.   
  As $|C(6)|=112$, it follows that $$a(C)=(1,0,0,0,0,0,112,a_7,a_8,a_7,112,0,0,0,0,0,1).$$ 
  We conclude from (\ref{kracheqn}) and \cite[Lemma 5.3.3]{vanlint}
  that the following constraints must
  hold: \begin{equation*}240-12a_7-8a_8 \geq 0;\,\,\,
-840-28a_7+28a_8\geq 0,\end{equation*}
   with $a_7\geq 0$ and $a_8\geq
  0$.  Solving these constraints gives that $a_7=0$ and $a_8=30$. 
    Consequently $C$ is a $(16,256,6)$ binary code, and so, by
    Snover's result \cite{snover}, $C$ is equivalent to the
    Nordstrom-Robinson code, proving the first part of Theorem \ref{mainthm}.  

  {\underline{\bf{Case $(m,\delta)=(15,5)$:}}}  Here $C(5)$ forms a
  $j-(15,5,\lambda_j)$ design for $j\leq 2$ with $\lambda=\lambda_2=4$, $\lambda_1=14$,
  and $\lambda_0=42$. As above let $\bs{\beta}\in C(5)$ and define $n_i=|\{\bs{\gamma}\in C(5)\,:\,|\supp(\bs{\gamma})\cap\supp(\bs{\beta})|=i\}|$.    
  Since $C(5)$ is a simple design with minimum distance $\delta=5$, we deduce that $n_5=1$, $n_4=n_3=0$.
  By applying \cite[Thm. 5]{semzin69balcodes}, we calculate that $n_2=30$, $n_1=5$ and $n_0=6$. 
  Since $n_0=6$, $a_{10}\neq 0$ in the distance distribution of $C$.
  Now, by following a similar argument to the one we used in the previous case, we deduce that $C$ is in fact antipodal, and
  so  $$a(C)=(1,0,0,0,0,42,a_6,a_7,a_7,a_6,42,0,0,0,0,1)$$Again, using the MacWilliams transform, we deduce that the following inequalities must hold:
  \begin{equation*} 630-6a_6-14a_7\geq 0;\,\,\,-210-6a_6+42a_7\geq 0;\,\,\,-390+6a_6-2a_7\geq 0,
\end{equation*}
  with $a_6\geq 0$, $a_7\geq 0.$ These solve to give $a_6=70$ and $a_7=15$, and so $C$ is a $(15,256,5)$ 
  binary code.  Thus, by Snover's result \cite{snover}, $C$ is
  equivalent to the punctured Nordstrom-Robinson code, proving the second part of Theorem
  \ref{mainthm}.
  
  By \cite[Lemma 2]{fpa}, complete transitivity is preserved by
  equivalence, and by Theorem \ref{nrctr} and Theorem \ref{pnrctr},
  the Nordstrom-Robinson codes are completely transitive.
  Consequently, in both cases, $C$ is completely transitive, proving the final statement of Theorem \ref{mainthm}.

\section{Nordstrom-Robinson Code as a $\Z_4$-linear code}\label{s:z4linear}
The Nordstrom-Robinson code is also the first member of another infinite family of non-linear binary codes, the \emph{Kerdock codes} \cite{kerdock}.
For each odd $k\geq 3$, the Kerdock code $\mathcal{K}(k)$ is a code of length $2^{k+1}$, with $\mathcal{K}(3)$ equal to $\mathcal{N}$. The codes $\mathcal{K}(k)$ and $\mathcal{P}(k)$ are \emph{formally dual}, by which we mean the distance distribution of one can
be obtained by taking the MacWilliams transform of the distance distribution of the other. In particular, the Nordstrom-Robinson code is \emph{formally self dual}. 
However, as these codes are non-linear, neither is the dual code of the other. It was not until work by Hammons et al.~\cite{hammons}
on linear codes over $\Z_4$ that an explanation for this phenomenon was discovered. 

To describe Hammons et al.~work, we first define the \emph{Lee metric}. We define 
$d_L(a,b)$ for $a,b\in\Z_4$ as follows: $d_L(a,b)=2$ if and only if $\{a,b\}=\{0,2\}$ or $\{1,3\}$, 
otherwise $d_L(a,b)=1$. We extend this definition to $m$-tuples of $\Z_4$, that is, the \emph{Lee distance} between 
${\bs{\alpha}},{\bs{\beta}}\in\Z_4^m$ is $$d_L(\bs{\alpha},\bs{\beta})=\sum_{i=1}^m d_L(\alpha_i,\beta_i).$$
We define the \emph{Grey map} to be the bijection $f:\Z_4\longmapsto \F_2^2$ given by
\begin{equation}\label{greymap}f(0)=00,\,\,f(1)=01,\,\,f(2)=11,\,\,f(3)=10,\end{equation}
and we extend this map to a bijection from $\Z_4^m$ to $\F_2^{2m}$ by 
$$\phi((\alpha_1,\ldots,\alpha_m))=(f(\alpha_1),\ldots,f(\alpha_m)).$$
The map $\phi$ is an isometry from $\Z_4^m$, with the Lee metric, to $\F_2^{2m}$, with the Hamming metric \cite[Thm.~1]{hammons}.

A \emph{linear code $C$ over $\Z_4$} of length $m$ is an additive subgroup of $\Z_4^m$.
An inner product on $\Z_4^m$ is defined to be $\bs{\alpha}\cdot\bs{\beta}=\alpha_1\beta_1+\ldots+\alpha_m\beta_m \mod 4$ from which
the usual notion of a \emph{dual code} $C^\perp$ can be defined.  
Hammons et al.~proved that the Kerdock codes and the Preparata codes of length $2^{k+1}$
are the image under $\phi$ of certain linear codes $\mathcal{C_K}$ and $\mathcal{C_P}$ in $\Z_4^m$, where $m=2^k$.
Moreover, these codes are dual codes of each other in $\Z_4^m$, that is $\mathcal{C_K}^{\perp}=\mathcal{C_P}$, explaining why the distance
distributions are related as they are.

Let $\Gamma$ be the graph with $V(\Gamma)=\Z_4^m$ and adjacency given by the Lee metric, that is,
$\bs{\alpha},\bs{\beta}\in V(\Gamma)$ are adjacent if and only if $d_L(\bs{\alpha},\bs{\beta})=1$. 
Since $\phi$ is a bijective isometry from $\Gamma$ to $H(2m,2)$, it follows that $\phi$ is a graph isomorphism. 
Therefore, $\Gamma$ and $H(2m,2)$ have isomorphic automorphism groups, namely $\Aut(\Gamma)\cong S_2\wr S_{2m}$.
Moreover, a code $C$ is completely transitive in $\Gamma$ if and only if it is completely transitive in $H(2m,2)$. Thus, we have the following.

\begin{proposition} The Octacode, the $\Z_4$-representation of the Nordstrom Robinson code, is completely transitive.
\end{proposition}

It is natural to ask if one can prove that a code in $\Gamma$ is completely transitive without appealing to its binary representation. 
Our interpretation of this question is that the symmetries involved in the proof should preserve the module structure of $\Z_4^m$. 
The largest subgroup of $\Aut(\Gamma)$ which preserves this structure is determined in the following lemma.

\begin{lemma}\label{l:z4autgrp} Let $\Gamma$ be defined as above.  Then the subgroup $G$ of $\Aut(\Gamma)$ that preserves the 
$\Z_4^m$ structure in $\Gamma$ is isomorphic to $D_8\wr S_m$.  
\end{lemma}

\begin{proof} Any automorphism of $\Gamma$ that preserves $\Z_4^m$ structure must preserve the partition 
$$\{\{1,2\},\{3,4\},\ldots,\{2m-1,2m\}\}$$
in its action on the vertex entries of $H(2m,2)$.
The largest subgroup of $\Aut(H(2m,2))$ that preserves this partition is $S_2\wr(S_2\wr S_m)$.  
Writing this as a subgroup of the wreath product acting on $\Z_4^m$, this is equal to $(S_2\wr S_2)\wr S_m$.  Now
$S_2\wr S_2=D_8$. Therefore the group $G$ of automorphisms of $\Gamma$ that preserve the $\Z_4^m$ structure is a subgroup
of $D_8\wr S_m$. 

Now let $H$ be the group generated by the permutations $(0,1,2,3)$ and $(0,2)$ of $\Z_4$, so $H\cong D_8$.
The group $H\wr S_m=H^m\rtimes S_m$ acts on the vertices of $\Z_4^m$ in its product action (similar to the action
of $S_2\wr S_{2m}$ on the vertices of the Hamming graph $H(2m,2)$ given in \eqref{eq:hamact}). It is clear that
$S_m$ preserves adjacency in $\Gamma$.  Moreover, by placing the elements of $\Z_4$ on 
the corners of a square, one deduces that $H$ preserves the Lee metric on $\Z_4$, and so $H^m$ preserves adjacency
in $\Gamma$.  Thus $H\wr S_m\leq G$.
\end{proof}

Our view of symmetry of a $\Z_4$-code $C$ allows all symmetries of $C$ in $\Aut(\Gamma)=S_2\wr S_{2m}$. 
Namely we consider the full symmetry group to be the setwise stabiliser of $C$ in $S_2 \wr S_{2m}$. Since $D_8 \wr S_m < S_2 \wr S_{2m}$, 
this group may be the same as the stabiliser of $C$ in $D_8\wr S_m$, or it may be larger.  If it is larger then there is the potential for the larger 
group to act completely transitively while the group preserving the $\Z_4$-structure does not.  Indeed this is the case for the 
Nordstrom-Robinson code and its $\Z_4$-representation the \emph{Octacode}.  That is to say, it is straightforward to show that the
stabiliser of the Octacode in $D_8\wr S_m$ is properly contained in its stabiliser in $\Aut(\Gamma)$ and does not act completely transitively on the code. 
Therefore, to prove the complete transitivity of the Nordstrom-Robinson code (and thus the Octacode), one should consider its binary representation.

\begin{acknowledgement} The authors would like to thank the anonymous referees for their helpful comments which improved the presentation of the paper.
\end{acknowledgement}


\def\cprime{$'$} \def\cprime{$'$}

\affiliationone{
   Neil I. Gillespie\\
   Heilbronn Institute for\\ Mathematical Research,\\ 
   School of Mathematics, Howard House,\\
   The University of Bristol,\\ Bristol, BS8 1SN,\\
   United Kingdom.
   \email{neil.gillespie@bristol.ac.uk}}```
\affiliationtwo{
   Cheryl E. Praeger\\
   Centre for the Mathematics of\\ Symmetry and Computation,\\
   School of Mathematics and Statistics,\\
The University of Western Australia,\\
35 Stirling Highway, Crawley,\\ Western Australia, 6009.\\
Also affiliated with\\ King Abdulaziz University, Jeddah,\\ Saudi Arabia.
   \email{cheryl.praeger@uwa.edu.au}}

\end{document}